\documentclass[10pt]{amsart}

\usepackage[utf8]{inputenc} 
\usepackage[T1]{fontenc}    
\usepackage{lmodern}        
\usepackage{amssymb}        
\usepackage{amsthm}         
\usepackage{mathtools}      
\usepackage{physics}        
\usepackage{xcolor}         
\usepackage[
    hyperfootnotes=false      
]{hyperref}                 
\usepackage{cleveref}       
\usepackage{microtype}      
\usepackage{enumitem}       
\usepackage{thm-restate}    
\usepackage{chngcntr}       
\usepackage{filecontents}   
\usepackage[
    maxbibnames = 99,         
    maxcitenames = 2,         
    backend = biber           
]{biblatex}                 

\makeatletter
\renewbibmacro*{related:bytranslator}[1]{%
  \entrydata{#1}{%
    \renewbibmacro*{name:hook}[1]{%
      \ifnumequal{\value{listcount}}{1}
        {\begingroup
         \mkrelatedstring%
         \lbx@initnamehook{#1}%
         \endgroup}
        {}}%
    \printnames[bytranslator]{translator}%
    \clearname{translator}%
    \setunit*{\addspace\bibstring[\mkrelatedstring]{astitle}\space}%
    \usebibmacro{title}%
    \ifentrytype{article}
      {\usebibmacro{in:}%
       \usebibmacro{journal+issuetitle}}
      {}
    \ifentrytype{inbook}
      {\usebibmacro{in:}%
       \usebibmacro{bybookauthor}%
       \newunit\newblock
       \usebibmacro{maintitle+booktitle}%
       \newunit\newblock
       \usebibmacro{byeditor+others}%
       \newunit\newblock
       \printfield{edition}}
      {}
    \ifentrytype{incollection}
      {\usebibmacro{in:}%
       \usebibmacro{maintitle+booktitle}%
       \newunit\newblock
       \usebibmacro{byeditor+others}%
       \newunit\newblock
       \printfield{edition}%
       \newunit
       \iffieldundef{maintitle}
         {\printfield{volume}%
          \printfield{part}}
         {}}
      {}
    \ifentrytype{inproceedings}
      {\usebibmacro{in:}%
       \usebibmacro{maintitle+booktitle}%
       \newunit\newblock
       \usebibmacro{event+venue+date}%
       \newunit\newblock
       \usebibmacro{byeditor+others}%
       \newunit\newblock
       \iffieldundef{maintitle}
         {\printfield{volume}%
          \printfield{part}}
         {}}
      {}
    \setunit{\addspace}%
    \printtext[parens]{%
      \printlist{location}%
      \iflistundef{publisher}
        {\setunit*{\addcomma\space}}
        {\setunit*{\addcolon\space}}%
      \printlist{publisher}%
      \setunit*{\addcomma\space}%
      \printdate}
      \newunit\newblock
      \printfield{note}%
      \newunit\newblock
      \printfield{chapter}%
      \setunit{\bibpagespunct}%
      \printfield{pages}
       \newunit\newblock  
       \printfield{url}%
      }}
\makeatother

\DeclareFieldFormat[article]{volume}{\mkbibbold{#1}}
\DeclareFieldFormat[article]{pages}{#1}
\DeclareFieldFormat[incollection]{pages}{#1}

\makeatletter
\def\l@section{\@tocline{1}{0pt}{0pc}{5pc}{}}
\def\l@subsection{\@tocline{2}{0pt}{2.5pc}{5pc}{}}
\makeatother

\definecolor{linkblue}{HTML}{00356B}
\definecolor{linkgold}{HTML}{DA9100}
\definecolor{linkred}{RGB}{159,  29, 53}
\hypersetup{
	colorlinks = true,
	linkcolor = linkred,
	citecolor = linkred,
	urlcolor = linkblue
}

\theoremstyle{plain}
\newtheorem{theorem}{Theorem}[section]
\crefname{theorem}{Theorem}{Theorems}

\crefname{conjecture}{Conjecture}{Conjectures}
\newtheorem{proposition}[theorem]{Proposition} 
\crefname{proposition}{Proposition}{Propositions}
\newtheorem{corollary}[theorem]{Corollary} 
\crefname{corollary}{Corollary}{Corollaries}
\newtheorem{lemma}[theorem]{Lemma} 
\crefname{lemma}{Lemma}{Lemmas}

\theoremstyle{definition}
\newtheorem{example}[theorem]{Example}
\crefname{example}{Example}{Examples}
\newtheorem{definition}[theorem]{Definition}
\crefname{definition}{Definition}{Definitions}

\theoremstyle{remark}
\newtheorem*{remark}{Remark}

\crefname{appendix}{Appendix}{Appendices}
\crefname{section}{Section}{Sections}
\crefname{figure}{Figure}{Figures}
\crefname{equation}{Equation}{Equations}

\counterwithin{equation}{theorem}

\newcommand{\field}[1]{\mathbf{#1}}
\newcommand{\R}{\field{R}}

\renewcommand{\H}{\field{H}}
\newcommand{\C}{\field{C}}

\newcommand{\poly}[1]{{#1}}
\newcommand{\hull}{H}

\def\titletext{Certain Hyperbolic Regular Polygonal Tiles are Isoperimetric}
\title[Certain Hyp. Reg. Polygonal Tiles are Isoperimetric]{\MakeUppercase{\titletext}}

\author[J. Hirsch]{Jack Hirsch}
\author[K. Li]{Kevin Li}
\author[J. Petty]{Jackson Petty}
\author[C. Xue]{Christopher Xue}

\address{Department of Mathematics \\ Yale University \\ New Haven, CT 06510, USA}
\email[Jack Hirsch]{jack.hirsch@yale.edu}
\email[Kevin Li]{k.li@yale.edu}
\email[Jackson Petty]{jackson.petty@yale.edu}
\email[Christopher Xue]{christopher.xue@yale.edu}

\date{\today}

\begin{document}

\begin{abstract}
The hexagon is the least-perimeter tile in the Euclidean plane. On hyperbolic surfaces, the isoperimetric problem differs for every given area. Cox conjectured that a regular $k$-gonal tile with 120-degree angles is 
isoperimetric for its area. We prove his conjecture and more. 
\end{abstract}

\maketitle
\tableofcontents

\section{Introduction}
In \citeyear{hales}, \textcite{hales} proved that the regular hexagon is the 
least-perimeter unit-area tile of the plane, and furthermore that
no such tiling of a flat torus is better (\cref{fig:hales}). Efforts to 
generalize this result to hyperbolic surfaces have been unsuccessful (see 
\cref{sect:monohedral}). We focus on monohedral tilings (by a single prototile)
and prove that a regular $k$-gon with $120^\circ$ angles is optimal 
(\cref{cor:monohedral-tilings}). 
Unlike Hales's deep proof, our result does not require computers. 

Our main 
\cref{thm:main-result} more generally treats multihedral tilings and varying 
areas averaging $A_k$. It proves that the \emph{maximum} perimeter of any tile in the tiling is greater than the perimeter $P_k$ of the regular $k$-gon $R_k$ with $120^\circ$ angles and area $A_k$. It extends to all real $k$ and 
hence all positive $A_k$.

\setcounter{theorem}{6}
\setcounter{section}{5}
\begin{restatable}{theorem}{mainresult}
\label{thm:main-result}
For real $k > 6$, consider a curvilinear polygonal tiling of a closed hyperbolic 
surface with $N$ tiles of average area $A_k$ and perimeter at most $P_k$. Then $k$ 
is an integer and every tile is equivalent to $\poly{R}_k$.
\end{restatable}
\setcounter{section}{1}

\subsection*{Methods}
To prove $\poly{R}_k$ is the optimal tile of an appropriate closed hyperbolic 
surface, \cref{prop:reg-is-best} first verifies that among $n$-gons of given 
area, the regular one minimizes perimeter. We seem to provide the first complete proof in the literature of this folk theorem, including 
\cref{isostribest} that the least-perimeter triangle of given area is isosceles.
It follows easily that $\poly{R}_k$ has less perimeter than all other $n$-gonal 
tiles for $n \leq k$. For $n>k$, \cref{lemma:vertexdegrees}, using the 
Gauss-Bonnet theorem, shows that in an $n$-gonal tiling, there are on average at 
most $k$ vertices of degree $3$ or more per tile.

The main difficulty concerns nonconvex tiles with many sides. Cutting corners saves perimeter, but the resulting shape does not necessarily tile. \cref{cover} shows that the convex hulls of each tile's vertices of degree at least $3$ cover the surface, albeit with polygons generally of unequal areas and variable number of sides.
By a new concavity \cref{lem:area-n}, $k$-gons would enclose more area with the same perimeter, exhibiting a $k$-gon better than the regular $k$-gon, a contradiction.

\textcite{hales} remarks that \citeauthor{toth43},  who proved the honeycomb conjecture for \emph{convex} tiles  \cite{toth43}, predicted considerable difficulties for general tiles \cite[183]{tothFigures} and said that the conjecture had resisted all attempts at proving it \cite{tothBees}. Removing the convexity hypothesis is the major advance of Hales's work and of ours, although we focus on polygonal monohedral tilings. It remains an open question whether a hyperbolic multihedral tiling with areas $A_k$ could have less \emph{average} perimeter than the regular polygon $R_k$ of area $A_k$ and angles $2\pi/3.$

\begin{figure}
\centering
\includegraphics[scale=0.5]{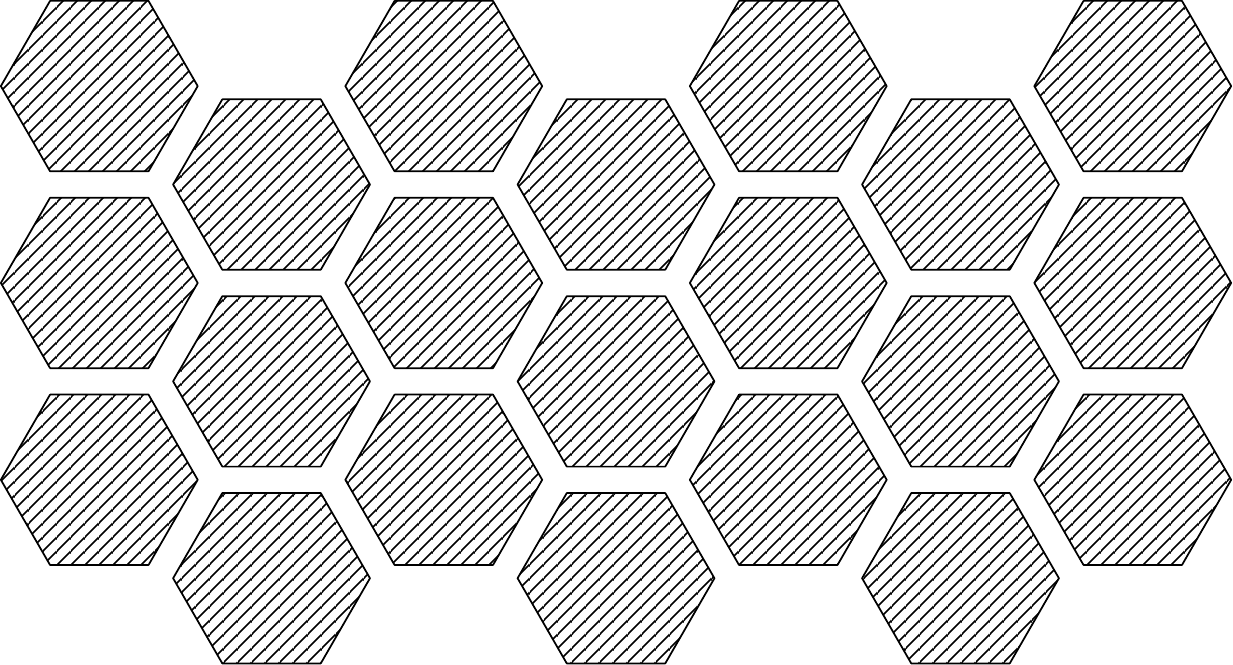}
\caption{\citeauthor{hales} (\citeyear{hales}) proved that regular hexagons provide the least-perimeter equal-area tiling of the plane.}
\label{fig:hales}
\end{figure}

\section{Definitions}

\begin{definition}[Tiling]
Let $M$ be a closed Riemannian surface. A tiling of $M$ is an
embedded multigraph on $M$ with no vertices of degree $0$ or $1$. A tiling is \emph{polygonal}
if 
\begin{enumerate}
\item\label{item:geodesic} every edge is a geodesic;
\item\label{item:disk} every face is an open topological disk.
\end{enumerate}
The oriented boundary of a face of a polygonal tiling is called a polygon.  A tiling is \emph{monohedral} if all faces are congruent. We sometimes consider \emph{curvilinear polygonal} tilings, relaxing condition \eqref{item:geodesic}.
\end{definition}

\begin{remark}
By definition our tilings are edge-to-edge.
When tiling a closed surface with a tile, one copy might be edge-to-edge with itself. An example is tiling a hyperbolic two-holed torus with a single octagon. All eight vertices join at one point, and each edge coincides with another edge. A second example is tiling a one-holed torus by tiling the square fundamental region with thin vertical rectangles. Each rectangle is edge-to-edge with itself at top and bottom, and the two vertices of a vertical edge coincide. 

All polygonal tilings are \emph{connected} multigraphs as a consequence of \eqref{item:disk}. 
\end{remark}

\begin{definition}[Equivalence]
Two polygons $\poly{Q}$ and $\poly{Q}'$ are \emph{equivalent} $\poly{Q}\sim\poly{Q}'$ \ if they are equal after the removal of all vertices of measure $\pi$.
\end{definition}

\begin{remark}
We can't in general define away vertices of measure $\pi$; a vertex in a tiling could, for example, have angles $\pi, \pi/2, \pi/2.$
\end{remark}

\begin{definition}[Convex Hull]
Let $R$ be a polygonal region on a closed hyperbolic surface $M$.
The convex hull $\hull (R)$ is taken in the hyperbolic plane
(with the minimal number of vertices). 
The convex hull of an $n$-gonal region $R$ is a $k$-gonal region for some~$k \leq n$. 
The convex hull has no less area and no more perimeter. 
\end{definition}

\begin{remark}[Existence] 
By standard compactness arguments, there is a perimeter-minimizing tiling for prescribed areas summing to the area of the surface, except that polygons may bump up against 
themselves and each other, possibly with angles of measure 0 and $2\pi$, in the limit.
We think that no such bumping occurs, but we have no proof. 
\end{remark}

\section{Hyperbolic Geometry}\label{sect:hyperbolic-geometry}

We begin with basic results of hyperbolic geometry.
\cref{prop:reg-is-best} seems to provide the first complete proof of the
folk theorem that the regular hyperbolic $n$-gon has least perimeter among all 
$n$-gons of fixed area, based on the fact that the best triangle of given base and area is isosceles (\cref{isostribest}). A key ingredient is the hyperbolic
Heron's formula (\cref{Heron}).
\cref{cor:combination-3-7}
proves that the regular $k$-gon is optimal among polygons with $k$ or fewer sides.


\begin{proposition}
\label{prop:Gauss-Bonnet}
By the Gauss-Bonnet formula, an $n$-gon in the hyperbolic plane with 
interior angles $\theta_1, \dots, \theta_n$ has area $(n-2)\pi-\sum  \theta_i$. In particular, a regular $n$-gon with interior angle $\theta$ has area given by
\begin{equation} \label{eq:A-n-theta}
A(n,\theta) = (n-2)\pi-n\theta. 
\end{equation}
\end{proposition}

\begin{proposition}[Law of Cosines]
\label{LoC}
If $\ell$ is the length of the side opposing angle $\theta_3$ in a triangle with interior angles $\theta_i$, then 
\[\cos{\theta_3} =\sin{\theta_1}\sin{\theta_2}\cosh{\ell}-\cos{\theta_1}\cos{\theta_2}.
\]
In particular, for right triangle $\triangle ABC$ with legs $a,b$, 
\[
\cosh{(a)}=\cos{(\angle A)}/\sin{(\angle B)}.
\]
\end{proposition} 

\begin{proposition}
\label{prop:perimeter-n-theta}
A regular $n$-gon with interior angle $\theta$ has perimeter given by
\begin{equation} \label{eq:P-n-theta}
P(n,\theta)=2n\cosh^{-1}\left(\frac{\cos(\pi/n)}{\sin(\theta/2)}\right).
\end{equation}
\end{proposition}

\begin{proof}
Connect the center of the regular $n$-gon to each of its vertices to form $n$ isosceles triangles. 
Bisect each triangle into two right triangles by connecting the center of the polygon to the midpoint of each side of the polygon.
Each triangle has interior angles $\pi/2, \pi/n$, and $\theta/2$. By \cref{LoC}, the length of the leg on the polygonal side of each of the $2n$ right triangles is $\cosh^{-1}(\cos(\pi/n)/\sin(\theta/2))$.
\end{proof}

\begin{definition} \label{def:AkPk}
For real $k>6$, let $A_k=A(k,2\pi/3)=(k-6)\pi/3$ and $P_k = P(k,2\pi/3)$, extending the area and perimeter of the regular $k$-gon $\poly{R}_k$ with angles $2\pi/3$ to real values of $k$.
Note that $A_k$ and $P_k$ increase from $0$ to $\infty$ as $k$ increases from $6$ to $\infty$.
\end{definition}

The hyperbolic version of Heron's formula gives the areas of
hyperbolic triangles in terms of their side lengths. 

\begin{proposition}[Heron's Formula, \cite{glasgow}] 
\label{Heron}
For a triangle in $\H^2$ with sides $x,y,z$, the area $A$ satisfies
\[
\tan^2 \frac{A}{2} = \frac{1 - \cosh^2 x - \cosh^2 y - \cosh^2 z + 2\cosh x \cosh y \cosh z}{(1 + \cosh x + \cosh y + \cosh z)^2}.
\]
\end{proposition}

\textcite{carroll06} provide the following simple proof that among hyperbolic $k$-gons of given area, the regular one minimizes perimeter. The previously published proof by \textcite{bezdek} used without proof the nontrivial fact (\cref{isostribest}) that for given base and area, an isosceles triangle minimizes perimeter. \citeauthor{carroll06} (Prop. 2.5) deduced this fact from Heron's formula, though their statement of Heron's formula was not quite right. In 2016, in discussions with Steve Openshaw, Colin Adams---unaware of the \citeauthor{carroll06} proof---produced a longer geometric proof (private communication).

\begin{lemma}
\label{isostribest}
For fixed base and perimeter, the isosceles triangle uniquely maximizes area in $\H^2$.
\end{lemma}

\begin{proof}
Consider a triangle with side lengths $x,y,z$. By \cref{Heron}, 
\[
\tan^2 \frac{A}{2} = \frac{1 - \cosh^2 x - \cosh^2 y - \cosh^2 z + 2\cosh x \cosh y \cosh z}{(1 + \cosh x + \cosh y + \cosh z)^2},
\]
where $A$ is area. Fixing the base $z$,
\begin{equation}
\label{eq:maxme}
\tan^2 \frac{A}{2} = \frac{a - \cosh^2 x - \cosh^2 y + 2 m \cosh x \cosh y}{(b + \cosh x + \cosh y)^2}
\end{equation}
for constants $a,b$, and $m = \cosh(z)$. Fix $x+y = 2c$, thereby fixing perimeter. It is possible to simultaneously maximize the numerator and minimize the denominator (which are both positive). The numerator is maximized by maximizing 
\[
F(x) = 2m\cosh x\cosh(2c-x) - \cosh^2x-\cosh^2(2c-x).
\]
A short computation and simplification makes the critical equation
\[
F'(x) = 4(\cosh(2c) - m)\sinh(c-x)\cosh(c-x) = 0.
\]
Observe $0<\cosh(c-x)$. Also, by the triangle inequality, $z<2c$ so $m < \cosh(2c)$. Thus $\sinh(c-x)=0$, which means $x=c$. This critical point is the unique global maximum as the derivative is positive for $0<x<c$ and negative for $c<x <2c.$

To minimize the denominator of \cref{eq:maxme}, set $x = c = y$. Therefore \cref{eq:maxme}, and thus area, is uniquely maximized for $x = y$, that is, when the triangle is isosceles.
\end{proof}

\begin{proposition}[\cite{carroll06}, Prop.\ 2.5]\label{prop:reg-is-best}
In the hyperbolic plane, the regular $n$-gon $\poly{Q}_n$ has less perimeter
than any other $n$-gon $\poly{Q}$ of the same area.
\end{proposition}

\begin{proof}
First we show that the optimal $n$-gon $\poly{Q}$ must be convex and equilateral. 
For fixed perimeter $P$, an area-maximizing $\poly{Q}$, as the convex hull of $n$ points, exists by a standard compactness argument. If it has fewer than $n$ vertices, place extra vertices on one of the sides.

By \cref{isostribest}, any two adjacent sides must be of equal length, ignoring the extra vertices. Now add one of the extra vertices. Repeating the argument with a segment bounded by the vertex and the following adjacent edge shows that there are no extra vertices. Therefore $\poly{Q}$ is a convex equilateral $n$-gon. 

Finally, assume $\poly{Q}$ is not regular. Inscribe the regular $n$-gon $\poly{Q}_n$ with the same edge lengths in a circle. Adding the little region between each edge of $\poly{Q}_n$ and the circle to each edge of $\poly{Q}$ would yield another region with the same perimeter as the circle and at least as much area, a contradiction. 
\end{proof}


The following monotonicity result is generalized to noninteger $n$ in \cref{lem:area-n}.

\begin{proposition} \label{prop:reg-decreasing}
The perimeter of a regular $n$-gon for a fixed area is decreasing as a 
function of $n$.
\end{proposition}

\begin{proof}
Let $\poly{Q}_n$ and $\poly{Q}_{n+1}$ be the regular polygons of a fixed
area with $n$ and $n+1$ sides. Let $\poly{Q}'_{n+1}$ be an $(n+1)$-gon
formed by adding a vertex of measure $\pi$ to $\poly{Q}_n$.
By \cref{prop:reg-is-best},
\[  
P(\poly{Q}_{n+1}) < P(\poly{Q}'_{n+1}) = P(\poly{Q}_n). \qedhere 
\]
\end{proof}

\begin{remark}
As expected, the perimeter of a regular $n$-gon of area $A$ is increasing as a function of $A$, for $0 < A < (n-2)\pi.$ Indeed, by \cref{prop:Gauss-Bonnet} and \cref{prop:perimeter-n-theta}, the perimeter of the $n$-gon is
\[
2n\cosh^{-1}\left(\frac{\cos\pi/n}{\sin(((n-2)\pi-A)/2n)}\right),
\] and it is increasing because $\cosh^{-1}$ and sine are increasing over $(0,\infty)$ and $(0, \pi/2)$, respectively. 
\end{remark}
\begin{corollary}
\label{prop:3-k}
The regular $k$-gon has less perimeter than 
any other $n$-gon of equal or greater area for $3 \le n \le k.$
\end{corollary}
\begin{proof}
The corollary follows immediately from \cref{prop:reg-is-best,prop:reg-decreasing}.
\end{proof}

\begin{corollary}
\label{cor:combination-3-7}
Tile a closed hyperbolic surface by polygons of equal area with $k$ or fewer sides. Then each of those tiles has perimeter at least that of the regular $k$-gon of the same area.
\end{corollary}

\begin{proof}
The corollary follows immediately from \cref{prop:3-k}.
\end{proof}






\section{Monohedral Tilings of Closed Hyperbolic Surfaces} 
\label{sect:monohedral}
In 2005, \textcite{cox2005, cox2011} and subsequently \textcite{sesum} proposed generalizing Hales's hexagonal isoperimetric inequality to prove that a
tiling by regular $k$-gons $\poly{R}_k$ with $120^\circ$ angles ($k \geq 7$) minimizes perimeter among all (possibly multihedral) tilings of an appropriate closed hyperbolic surface.
\textcite{carroll06} showed that their proposed polygonal isoperimetric 
inequality fails for $k > 66$. 
\cref{cor:monohedral-tilings} independently proves $R_k$ optimal for \emph{monohedral} tilings.
Although \cref{cor:monohedral-tilings} applies even if the regular polygon does not tile, \cref{thm:existence-of-tilings} shows there are many closed hyperbolic surfaces which it does tile. 
It is possible for many-sided polygons to tile, but \cref{cor:3} shows that
as $n$ increases, $n$-gonal tiles necessarily have many concave angles.
\cref{cor:heptagon-beats-convex} deduces that the regular polygon has less perimeter than any other \emph{convex} polygonal tile.

\begin{remark}
By Gauss-Bonnet, the regular $k$-gon $\poly{R}_k$ of area $A_k = (k-6)\pi/3$ ($k \geq 7$) has interior angles of $2\pi/3$ (\cref{sect:hyperbolic-geometry}). It therefore tiles $\H^2$, as well as many closed hyperbolic surfaces (\cref{thm:existence-of-tilings}). For area not a multiple of $\pi/3$, there is no conjectured isoperimetric tile. 

That every other tile of the same area has more perimeter than $\poly{R}_k$ was known in the special case that the surface has area $A_k$, so that a single tile covers the whole surface. \textcite[653]{choe1989} proved the existence of such an isoperimetric single tile and shows that it is a polygon with $120^{\circ}$ interior angles. For example, the isoperimetric single tile on a flat torus is a $120^{\circ}$-angle hexagon (not a parallelogram) and always has at least the perimeter of the regular hexagon. On a closed hyperbolic surface of genus $g$, the isoperimetric single tile $T$ is a $120^{\circ}$-angle $(12g-6)$-gon and always has at least the perimeter of the regular $(12g-6)$-gon.
\end{remark}

\begin{proposition}\label{thm:existence-of-tilings}
For $k\geq 7$, there exist infinitely many closed hyperbolic surfaces tiled by the regular $k$-gon $R_k$ of area $A_k = (k-6)\pi/3$ and angles $2\pi/3$.
\end{proposition}

\begin{proof}
These surfaces are provided by work of \textcite[Main Thm.]{edmonds-tess} on torsion-free subgroups of Fuchsian groups and tessellations (see also \cite{edmonds-bull, edmonds}). Their work yields torsion-free subgroups $S$ of arbitrarily large finite index of the triangle group $(2, 3, k).$ This triangle group is the orientation-preserving symmetry group of the hyperbolic triangle of angles $\pi/2, \pi/3,$ and $\pi/k$. Each quotient of $\H^2$ by such a subgroup $S$ is a closed hyperbolic surface tiled by these triangles, which can be joined in groups of $2k$ to form a tiling by the regular $k$-gon of area $(k-6)\pi/3$ and hence angles $2\pi/3$ (by Gauss-Bonnet).
\end{proof}

\begin{example}
The Klein Quartic Curve in $\C P^2$ \cite{klein-original} is the set of complex solutions to the homogeneous equation
\[
    u^3v + v^3w + w^3u = 0.
\]
The curve is a hyperbolic 3-holed torus. It is famously tiled by 24 regular heptagons.
\end{example}

The following results are instrumental in eliminating competing $n$-gons of large~$n$.
\begin{lemma} \label{lemma:vertexdegrees}
Consider a tiling of a closed hyperbolic surface by curvilinear polygons $\poly{Q}_i$ of average area $A_k=(k-6)\pi/3$ for some real $k>6$. Then each polygon has on average at most $k$ vertices of degree at least $3$, with equality if and only if every vertex has degree two or three.
\end{lemma}

\begin{proof}
A tile with $n$ edges and $v$ vertices of degree at least $3$ contributes to the tiling $1$ face, $n/2$ edges, and at most $(n-v)/2 + v/3$ vertices, with equality precisely if no vertices have degree greater than $3$. Therefore it adds at most $1-v/6$ to the Euler characteristic $F-E+V$. The Gauss-Bonnet theorem says that 
\[ 
    \int G=2\pi(F-E+V). 
\]
Hence the average contributions per tile satisfy
\[
-A_k = -(k-6)\pi/3 \leq 2\pi(1-\overline{v}/6).
\]
Therefore $\overline{v} \leq k$, with equality if and only if no vertices have degree more than~$3$.
\end{proof}

\begin{proposition}
\label{cor:3}
Let $\poly{Q}$ be an $n$-gon of arbitrary area $A_k=(k-6)\pi/3$ (real $k>6$) with $\ell_1$ (interior) 
angles of measure $\pi$ and $\ell_2$ of measure greater than $\pi$. 
If $\poly{Q}$ tiles a closed hyperbolic surface $M$, then $\ell_1 + 2\ell_2 \ge n-k$. 
Equality holds for a tiling (and therefore every tiling) if and only if every vertex is of degree two or three, and every concave angle has degree two.
\end{proposition}

\begin{proof}
Take any tiling of $M$ by $\poly{Q}$. Each vertex of degree two in the tiling has either two angles of measure $\pi$ or exactly one angle of measure greater 
than $\pi$.
By Lemma~\ref{lemma:vertexdegrees},
$$\ell_1 + 2\ell_2 \ge n-k,$$
with equality precisely when every vertex has degree two or three, and every 
concave angle has degree 2.
\end{proof}

The following corollary proves \cref{cor:monohedral-tilings} among \emph{convex} polygonal tiles. 
\begin{corollary}\label{cor:heptagon-beats-convex}
The regular $k$-gon $\poly{R}_k$ has less perimeter than any non-equivalent convex polygonal tile of area $A_k = (k-6)\pi/3$.
\end{corollary}

\begin{proof}
Let $\poly{Q}$ be a convex $n$-gonal tile of area $A_k$. By \cref{cor:3}, $\poly{Q}$ contains at least $n-k$ angles of measure $\pi$. Hence $\poly{Q}$ is equivalent to a polygon with at most $k$ sides. Unless $\poly{Q}$ is equivalent to $\poly{R}_k$, $\poly{Q}$ has strictly more perimeter by \cref{prop:3-k}.
\end{proof}

\begin{remark}
Although it is easy to show that an isoperimetric curvilinear triangular tile must actually be polygonal by straightening the edges, 
an extension to all curvilinear $k$-gons remains conjectural because straightening one edge of a tile might cause it to intersect another part of the tile. 
\end{remark}

\section{Regular Polygonal Tiles are Isoperimetric}

Our main result, \cref{thm:main-result}, proves that regular $k$-gons $\poly{R}_k$ of area $A_k=(k-6)\pi/3$ (with $120^\circ$ angles and perimeter $P_k$) are optimal, even when they don't tile. It provides similar estimates for interpolated areas.
It also allows for multihedral tilings, showing that the maximum perimeter of such tiles is greater than or equal to $P_k$.

The main difficulty concerns nonconvex tiles with many sides. Cutting corners saves perimeter, but the resulting shape does not necessarily tile. \cref{cover} shows that the collection of convex hulls of each tile's vertices of degree at least $3$ covers the surface, although generally with polygons of unequal areas and variable number of sides. Fortunately, by Gauss-Bonnet, the average number of sides is at most $k$ (\cref{cover}).
By a new concavity \cref{lem:area-n}, the $k$-gons enclose more area with the same perimeter, exhibiting a $k$-gon better than $\poly{R}_k$, a contradiction.

To ensure that the convex hulls of the high-degree vertices cover, we start with straightening and flattening processes for curvy edges and degree-$2$ vertices.

\begin{definition}[Flattening]
Consider a polygonal chain $ABC$ in $\textbf{H}^2$. To flatten vertex $B$ is to replace $ABC$ with the geodesic $AC$. For a hyperbolic surface, flattening is done in the universal cover $\H^2$.
\end{definition} 

\begin{lemma}
\label{lemma:flatten-immersed}
Consider immersed curvilinear polygons $\poly{P}$ and $\poly{Q}$ in a hyperbolic surface that share either a vertex $V$ and the incident edges or an edge. Replacing the edge with a geodesic or flattening $V$ in the covering $\H^2$ yields immersed curvilinear polygons whose union contains $\poly{P}$ and $\poly{Q}$.  
\end{lemma}
\begin{proof}
Let $A$ and $B$ be the adjacent vertices of $V$. Let $R$ be the region enclosed by the new geodesic and the edges it replaced. Note that the union of of the resulting polygons is simply $P \cup Q \cup R$. 
\end{proof}

\begin{proposition}\label{cover}
Let $M$ be a closed hyperbolic surface tiled by curvilinear polygons $\poly{Q}_i$ of average area $A_k = (k-6)\pi/3$ for real $k>6$. Let $\poly{Q}_i^*$ be the convex hull of the vertices of degree three or higher of $\poly{Q}_i$. Then $\{\poly{Q}_i^*\}$ covers $M$ and the average number of sides is less than or equal to $k$. 
\end{proposition}
\begin{proof}
By \cref{lemma:flatten-immersed}, straightening edges and flattening all degree-$2$ vertices yields a covering by immersed polygons, each covered by the corresponding $\poly{Q}_i^*$. Hence $\{\poly{Q}_i^*\}$ covers~$M$. By \cref{lemma:vertexdegrees},
the average number of sides is less than or equal to~$k$.
\end{proof}
\begin{remark}
\label{lemma:atleasttwo}
For fixed $n$, every tile in a tiling by curvilinear $n$-gons of a connected closed surface, other than a sphere or $\R P^2$, has at least two vertices of degree at least $3$. Indeed, suppose a tile has fewer than two vertices of degree at least $3$. Such a tile must share all edges with itself or another tile (and actually has no vertices of degree at least $3$). Since the surface is connected, there are no other tiles, and the surface is a sphere or $\R P^2$. 
\end{remark}

\begin{remark}
\cref{fig:multibad} illustrates an unbounded example in which the convex hulls of each tile's vertices of degree at least three do not cover the surface.
\end{remark}

\begin{figure}
\center
\includegraphics{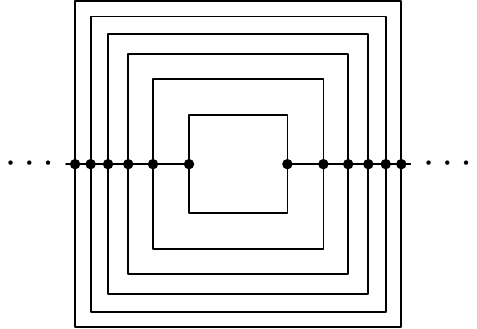}
\caption{ A tiling of the Euclidean plane by polygons of equal area, in which all vertices of degree three or more (here marked by dots) are collinear. The convex hull of these vertices is just a line, and certainly does not cover the plane.}
\label{fig:multibad}
\end{figure}

The concavity of the following area function for fixed perimeter is a crucial ingredient in the proof of the main result, \cref{thm:main-result}.

\begin{lemma} \label{lem:area-n}
The area of the regular $n$-gon with perimeter $P$ is given by
\begin{equation}
\label{eq:A(n)}
A(n) = 
\pi(n-2)-2 n \sin ^{-1}\left(\cos\alpha \sech\beta\right)
\end{equation}
where $\alpha = \pi/n$ and $\beta=P/2n$.
The function $A(n)$ is strictly increasing and strictly concave on $[2,\infty)$.
We extend $A(n)$ continuously to be identically $0$ on the interval $[0,2]$.
\end{lemma}

\begin{remark}
For nonintegral $n$, \cref{eq:A(n)} still holds when Equations~\eqref{eq:A-n-theta} and~\eqref{eq:P-n-theta} for $A(n,\theta)$ and $P(n,\theta)$ hold.
\end{remark}

\begin{proof}
By \cref{prop:perimeter-n-theta}, the perimeter of a regular $n$-gon with interior angle $\theta$ is given by
\[
P(n,\theta)=2n\cosh^{-1}\left(\frac{\cos(\pi/n)}{\sin(\theta/2)}\right),
\]
which increases from $0$ to $\infty$ as $\theta$ decreases from $(n-2)\pi/n$ to $0$. Solve for $\theta$ in the range $0<\theta<\pi$ to find
\[
\theta = 2\sin^{-1}\left(\cos\alpha\sech\beta\right).
\]
\cref{eq:A(n)} now follows from \cref{prop:Gauss-Bonnet}\eqref{eq:A-n-theta}.
To prove that $A(n)$ is strictly concave, remove a trivially negative factor from the second derivative $A''(n)$ and simplify, reducing the problem to showing that
\[
P^2 \tanh ^2\beta + \left(4 \pi ^2-P^2\right) \sech^2\beta + 
P^2 \cos ^2\alpha \sech^4\beta + 4 \pi  P \tan \alpha \tanh \beta - 4\pi^2 
\]
is positive for $n>2$ and $P>0$.
Substituting $T=\tan\alpha$ and $H=\tanh\beta$, rearranging terms, and simplifying give that it is sufficient to prove 
\[
P\left(H^2+T^2\right) - \sqrt{1+T^2}\cdot \left(PT-2\pi H\right)
\]
is positive. Since it vanishes at $P=0$, it suffices to show that the derivative with respect to $P$,
\[
\tanh^2(\alpha\beta)+\tan^2(\alpha)+2\alpha\beta\tanh(\alpha\beta)\sech^2(\alpha\beta)
-
\sec(\alpha)\left(\tan(\alpha) - \alpha\sech^2(\alpha\beta)\right),
\]
is positive for $0<\alpha<\pi/2$ and $\beta>0$. Substituting $c=\tanh^2(\alpha\beta)$ and simplifying reduces to showing that 
\[
c+\frac{\alpha}{\cos\alpha}(1-c) > \frac{\sin\alpha}{1+\sin\alpha}
\]
for $0<c<1$ and $0<\alpha<\pi/2$, 
which holds trivially.

Finally, strict monotonicity of $A(n)$ follows from strict concavity, since $A(n)$ remains positive for $n>2$.
\end{proof}

The following lemma and corollary are needed in the proof of the main \cref{thm:main-result} to handle the interval $[0,2)$ not covered by \cref{lem:area-n}.

\begin{lemma} \label{lemma:A-2A(k/2)}
Fix real $k >6.$ Consider $A(n)$ with fixed perimeter $P_k$. Then
\[A(k) < 2A\left(\frac{k}{2}\right).\]
\end{lemma}

\begin{proof}

Let $\gamma = \cos(\pi/k),$ so $\sqrt{3}/2<\gamma<1.$ By \cref{eq:P-n-theta},
\begin{align*}
\cosh\left(\frac{P}{2k}\right) 
&= \frac{\cos (\pi/k)}{\sin (\pi/3)}
= \frac{2\gamma}{\sqrt{3}}.
\intertext{By \cref{eq:A-n-theta} and the double angle identities,}
A(k) & =(k-2)\pi - \frac{2k\pi}{3}, \\
2A\left(\frac{k}{2}\right) &= (k-4)\pi - 2k\sin^{-1}\left(\frac{2\gamma^2-1}{8\gamma^2/3-1}\right).
\end{align*}
Algebraic manipulation shows the desired inequality is
\[
\sin^{-1}\left(\frac{2\gamma^2-1}{8\gamma^2/3-1}\right) < 
\frac{\pi}{3}-\frac{\pi}{k}.
\]
Both sides lie in the interval $[-\pi/2,\pi/2]$, over which sine is increasing. Thus it is equivalent to show
\[
\left(\frac{2\gamma^2-1}{8\gamma^2/3-1}\right) < \sin\left(\frac{\pi}{3}-\frac{\pi}{k}\right)=\frac{1}{2}\left(\gamma\sqrt{3}-\sqrt{1-\gamma^2}\right).
\]
Equality is attained at $\gamma=\cos(\pi/6)=\sqrt{3}/2,$ and the inequality is trivial at $\gamma=1.$ It thus suffices to show equality is never attained in $(\sqrt{3}/2,1)$; there are many ways to do so, one of which we use here. After rearrangement, equality holds only at the roots of the equation
\[\left(2(2\gamma^2-1)-\gamma\sqrt{3}(8\alpha^2/3-1)\right)^2 = \left( (8\gamma^2/3-1)\cdot \sqrt{1-\gamma^2}\right)^2,\] and so only at the roots of the sixth degree polynomial 
\[256 \gamma ^6-192 \sqrt{3} \gamma ^5-112 \gamma ^4+168 \sqrt{3} \gamma ^3-60 \gamma ^2-36
\sqrt{3} \gamma +27.\]
The first through sixth derivatives of this polynomial, evaluated at $\gamma=\sqrt{3}/2,$ are all positive:
\[
6\sqrt{3}, \quad 384, \quad 2554\sqrt{3},\quad 31872,\quad 69120\sqrt{3},\quad 184320.
\]
Since the sixth derivative is constant, they remain positive. Hence, equality is never attained in $(\sqrt{3}/2,1)$, and so the desired strict inequality for $k>6$ follows.
\end{proof}

\begin{corollary}
\label{A(n/2)}
Fix real $k>6.$ Consider $A(n)$ with fixed perimeter $P_k$. For all real $n \ge k,$
\[A(n)< 2A\left(\frac{n}{2}\right).\]
\end{corollary}
\begin{proof}
By \cref{lemma:A-2A(k/2)},
\[A(k) < 2A\left(\frac{k}{2}\right).\]
Since $A$ is strictly concave on $[2,\infty) \supset [k/2,\infty)$ and is strictly increasing,
\begin{align*}
A(n) &= A(k) + \left(A(n)-A(k)\right) \\
&< 2A\left(\frac{k}{2}\right) + 2\left(A\left(\frac{n}{2}\right)-A\left(\frac{k}{2}\right)\right) \\
&= 2A\left(\frac{n}{2}\right). \qedhere
\end{align*}
\end{proof}

Recall that $A_k$ and $P_k$ are the area and perimeter of the regular polygon $\poly{R}_k$ with $120^\circ$ angles, extended formulaically to all real $k > 6$ and increasing in $k$ (\cref{def:AkPk}). Our main theorem shows that as $k$ ranges from $6$ to $\infty$ and the average area $A_k$ ranges from $0$ to $\infty$, some tile must have perimeter at least $P_k$, with equality only if $k$ is an integer and every tile is equivalent to the regular $k$-gon $\poly{R}_k$.

\mainresult* \addtocounter{theorem}{1}

\begin{proof}
By \cref{cover}, the collection of convex hulls $\poly{Q}_i^*$ of the vertices with degree at least $3$ on each tile covers $M$, and of course $P(\poly{Q}_i^*) \leq P(\poly{Q}_i) \leq P_k$ by assumption.  
Since the $\poly{Q}_i^*$ cover, 
\begin{equation}
\label{ave-area}
\frac{1}{N}\sum \text{Area}(\poly{Q}_i^*) \ge A_k.
\end{equation}
By \cref{cover}, the number of sides $n_i$ of $Q_i^*$ satisfy 
\[\frac{1}{N}\sum n_i \le k.\]
The areas can be estimated in terms of $A(n)$ for $P_k$ as
\begin{equation}
\sum \text{Area}(\poly{Q}_i^*) \le \sum A(n_i) \le
N \cdot A \left(\frac{\sum n_i}{N} \right) \le 
N \cdot A(k) = 
N \cdot A_k.
\label{eq:avehullarea}
\end{equation}
The first inequality follows from \cref{prop:reg-is-best} and the remark after \cref{prop:reg-decreasing}.
The second inequality follows from the concavity of $A(n)$ for $n\geq 2$ (\cref{lem:area-n}) and Jensen's inequality. 
If any of the $n_i$ are $0$ or $1$, choose some $n_i > k$, and use \cref{A(n/2)} first to replace $0+A(n_i)$ with $2A(n_i/2)$. 
If you run out of large enough $n_i$, the next inequality holds already.
The third inequality follows from the fact that $A(n)$ is strictly increasing (again \cref{lem:area-n}). The final equality holds by the definition of $A(n)$ for $P_k$.

By \cref{ave-area}, equality must hold in every inequality.
By the strict concavity of $A(n)$, equality in the second inequality implies that every $n_i = k$, which must therefore be an integer. Equality in the first inequality implies that every $\poly{Q}_i^*$ has area $A$. By \cref{prop:reg-is-best}, $\poly{Q}_i^*$ is the regular $k$-gon $\poly{R}_k$ of area $A_k$. Finally 
\[P(\poly{Q}_i) \ge P(\poly{Q}_i^*) = P_k,\]
and equality implies that $\poly{Q}_i \sim \poly{R}_k$.
\end{proof}

\cref{thm:main-result} immediately implies the following corollary on monohedral tilings.

\begin{corollary}[Monohedral Tilings]
\label{cor:monohedral-tilings}
For $k \geq 7$, any non-equivalent tile of area $A_k = (k-6)\pi/3$ of a closed hyperbolic surface has more perimeter than the regular $k$-gon $\poly{R}_k$ (whether or not $\poly{R}_k$ tiles). \end{corollary}

\begin{remark}
It remains an open question whether \cref{cor:monohedral-tilings} extends to the hyperbolic plane, where matching discrepancies might be pushed off to infinity. Similarly considering large regions does not work, because truncation effects are too large.
\end{remark}

The following proposition shows  that in some sense the area of the regular hexagon $R_k$ increases more rapidly than the perimeter as the number of sides increases.

\begin{proposition}[Perimeter Ratio]
\label{PerimeterRatio}
For real $k>6$, $P_k/A_k$ is a (strictly) decreasing function of $k$.
\end{proposition}

\begin{proof}
By \cref{prop:Gauss-Bonnet} and \cref{prop:perimeter-n-theta}, in terms of $x=\pi/6-\pi/k$,
\[\frac{\pi^2}{9}\frac{A_k}{P_k}= \frac{x}{ \cosh^{-1}\left(\frac{\cos(\pi/6-x)}{\sin(\pi/3)}\right)}.\]
It suffices to show that the right hand side is strictly increasing in $x$ for $0<x<1$. 
By Wolfram Alpha, its derivative is given by
   \[\frac{x \sin \left(\frac{\pi }{6}-x\right)}{\sin(\pi/3)\cos ^{-1}\left(\frac{\cos
   \left(\frac{\pi }{6}-x\right)}{\sin(\pi/3)}\right)^2 \sqrt{\frac{\cos
   \left(\frac{\pi }{6}-x\right)}{\sin(\pi/3)}-1} \sqrt{\frac{\cos \left(\frac{\pi
   }{6}-x\right)}{\sin(\pi/3)}+1}}+\frac{1}{\cosh ^{-1}\left(\frac{ \cos
   \left(\frac{\pi }{6}-x\right)}{\sin(\pi/3)}\right)},\]
    which is positive for $0<x<1.$
    
    Hence, $P_k/A_k$ is strictly decreasing for $k>6$. 

\end{proof}

The following corollary shows in particular that reducing the area per tile of a monohedral tiling increases the total perimeter.

\begin{corollary}[Total Perimeter]
A tiling of a closed hyperbolic surface by $\poly{R}_k$ has less total perimeter than any nonequivalent tiling by polygons of equal perimeter $P$ and average area $A_m \le A_k$. 
\end{corollary}

\begin{proof}
Let $A$ denote the total area of the surface. By \cref{thm:main-result} and \cref{PerimeterRatio}, the competing total perimeter is greater than or equal to the total perimeter of the $\poly{R}_k$ tiling:
$$P\frac{A}{A_m} \geq P_m\frac{A}{A_m} \geq P_k\frac{A}{A_k},$$

\medskip
\noindent with equality only if all the tiles are equivalent to $\poly{R}_k$. 
\end{proof}

Our methods more easily yield the following weak version of \citeauthor{hales}'s hexagonal honeycomb theorem \cite{hales}. For details, see Proposition 10.5 of \citeauthor{digiosia}.

\begin{proposition}[Euclidean Hexagons]
Consider a curvilinear polygonal tiling of a flat torus with tiles of average area $A$. Then some tile has at least as much perimeter as the regular hexagon $R_6$ of area $A$, with equality only if every tile is equivalent to $R_6$.
\end{proposition}

\section*{Acknowledgements} This work is a product of the 2019 Summer Undergraduate Mathematics Research program at Yale (SUMRY) under the guidance of Frank Morgan of Williams College. The authors greatly thank Morgan for his help and insight over the many weeks spent researching and writing this paper. We thank the Young Mathematicians Conference (YMC) and Yale for supporting our trip to present at the 2019 YMC in Columbus, Ohio.

\newpage

\printbibliography

\end{document}